\documentclass[11pt]{article}
\usepackage{amsmath,amssymb,amsthm,graphicx,subfigure,float,url}

\usepackage{tikz}

\usepackage{parskip} % pour éviter l'indentation !
\usepackage[colorlinks=true]{hyperref}
\usepackage{pdfsync}
\graphicspath{{fig/}}
\usepackage{enumitem}
\usepackage{dsfont} %for indicator function characteric function
\usepackage[T1]{fontenc} % Use 8-bit encoding that has 256 glyphs

\def\R{\textrm{I\kern-0.21emR}}
\def\N{\textrm{I\kern-0.21emN}}

\renewcommand{\geq}{\geqslant}
\renewcommand{\leq}{\leqslant}

\renewcommand{\geq}{\geqslant}
\renewcommand{\leq}{\leqslant}

\newcommand {\Chi} {{\bf \raise 2pt \hbox{$\chi$}} }
\newcommand{\beq}{\begin{equation}}
\newcommand{\eeq}{\end{equation}}
\newcommand{\bea} {\begin{array}{rl}}
\newcommand{\eea} {\end{array}}
\newcommand{\bepa}{\left\{ \begin{array}{l}}
\newcommand{\eepa} {\end{array}\right.}

\newcommand{\bmu}{\begin{multline}}
\newcommand{\emu}{\end{multline}}

\newtheorem{theorem}{Theorem}

\theoremstyle{definition}
\title{On the stability of the state 1 in the non-local Fisher-KPP equation in bounded domains}

\author{
Camille Pouchol\thanks{Sorbonne Universit\'es, UPMC Univ Paris 06, CNRS UMR 7598, Laboratoire Jacques-Louis Lions, F-75005, Paris, France} \thanks{INRIA Team Mamba, INRIA Paris, 2 rue Simone Iff, CS 42112, 75589 Paris, France}  \thanks{e-mail: \href{mailto:pouchol@ljll.math.upmc.fr}{pouchol@ljll.math.upmc.fr}}
}
\date{}

\begin{document}
\maketitle
\pagestyle{plain}
%\tableofcontents
\pagenumbering{arabic}

\begin{abstract} 
We consider the non-local Fisher-KPP equation on a bounded domain with Neumann boundary conditions. Thanks to a Lyapunov function, we prove that under a general hypothesis on the Kernel involved in the non-local term, the homogenous steady state $1$ is globally asymptotically stable. This assumption happens to be linked to some conditions given in the literature, which ensure that travelling waves link $0$ to $1$.
\end{abstract} 

\section{Introduction}
We consider the so-called non-local Fisher-KPP equation endowed with Neumann boundary conditons
\begin{equation}
\begin{split}
\label{Model}
\dfrac{\partial u}{\partial t} (t,x) & =\mu \left(1 - \int_\Omega K(x,y) u(t,y) \, dy\right) u(t,x) + \Delta u (t,x), \; \; x \in \Omega, \; t>0,  \\
\dfrac{\partial u}{\partial n} (t,x) & = 0,  \; \; x \in \partial \Omega, \; t>0,\\
u(0,x) &= u^0(x) \geq 0 \; \; x \in \Omega, 
\end{split}
\end{equation}
where $\Omega$ a regular bounded domain of $\mathbb{R}^d$ and $K > 0$ a Kernel modelling an additional death rate due to non-local interactions. \par
We will sometimes write in short $K[u] = \int_\Omega K(x,y) u(y) \,dy$ for a generic function $u$.

Assuming 
\begin{equation}
\label{SumToOne}
\forall y \in \Omega, \; \int_\Omega K(x,y) \, dx = 1,
\end{equation}
and in the limit 
$K(x,y) \rightarrow \delta_{x-y}$, we recover the classical Fisher KPP-equation 
\begin{equation}
\label{Classical}
\dfrac{\partial u}{\partial t} =\mu (1 - u) u + \Delta u.
\end{equation}
The assumption \eqref{SumToOne} ensures that $1$ remains a homogeneous stationary solution of \eqref{Model}. 

The classical Fisher-KPP equation \eqref{Classical} is often analysed on the whole space for the investigation of travelling waves, which are known to exist since the pioneering works of Fisher, Kolmogorov, Petrovsky and Piskunov~\cite{Kolmogorov1937} for any speed above $2\sqrt{\mu}$. Furthermore, any-non zero initial condition eventually converges locally uniformly to $1$, which is therefore a globally asymptotically stable for non zero initial conditions.  

When one adds a non-local term, it does not remain true that travelling waves exist and when they do, whether they link $0$ to $1$ or another non-homogeneous steady-state of the equation. $1$ can indeed become unstable: Türing patterns appear~\cite{Nadin2011, Perthame2015}. 

A natural question is thus to understand under which conditions the status of $1$ is changed due to the non-local term. When $K(x,y)$ is given by a convolution $\phi(x-y)$, several results have already been obtained in the full space, in dimension $d=1$. If the Fourier transform is everywhere positive of if $\mu$ is small enough, it is known that travelling waves necessarily connect $0$ to $1$~\cite{Berestycki2009}. See also~\cite{Alfaro2012, Hamel2014}

In this note, we provide a general result on the global asymptotic stability on $1$ on a bounded domain, based on a Lyapunov functional. The results holds provided that the following general assumption on the Kernel $K$ is satisfied: 
\begin{equation}
\label{Positive}
\forall f \in L^2(\Omega), \;  \int_{\Omega\times \Omega} K(x,y) f(x) f(y)\, dx \, dy \geq 0.
\end{equation}
$K$ is then referred to as being a positive Kernel, and \eqref{Positive} can be thought of as a strong competition assumption. 
These types of Lyapunov functionals have been used successfully in selection equations in~\cite{Jabin2011, Pouchol2017a, Pouchol2017} and are inspired by Lyapunov functions for Lotka-Volterra ODEs~\cite{Goh1977}.

It remains an open question to know whether this condition leads to the same conclusion on the whole space. As such, our Lyapunov function requires integrability for $u(t) - 1 - \ln(u(t))$ which is too much to ask in $\mathbb{R}^d$. We still believe that the condition \eqref{Positive} is highly relevant. Indeed, when $\Omega = \mathbb{R}^d$, and if $K$ is a convolution $K(x,y) =\phi(x-y)$, then condition \eqref{Positive} becomes 
\begin{equation}
\label{PositiveAgain}
\forall f \in L^2\big( \mathbb{R}^d\big), \;  \int_{ \mathbb{R}^d\times  \mathbb{R}^d} \phi(x-y) f(x) f(y)\, dx \, dy \geq 0.
\end{equation}

It is easy to check that if $\phi$ has a non-negative Fourier transform, then condition \label{PositiveAgain} is satisfied, see~\cite{Jabin2011}. 
The converse is almost true, as evidenced by Bochner's Theorem~\cite{Reed1975}: if $\phi$ is bounded an continuous, then \eqref{Positive} holds if and only if it is the Fourier transform of a finite bounded measure on $\mathbb{R}^d$. 

Consequently, condition \eqref{Positive} (or \eqref{PositiveAgain}) shows that the condition on the Fourier transform of $\phi$ used in dimension $1$ in the literature can be appropriate in any dimension, and may not only be a sufficient but also a necessary condition when it comes to the stability of the state $1$.

\section{The Lyapunov function approach}

We make the following regularity assumption on the Kernel $K$: 
\begin{equation}
\label{Regularity}
K \in C^{0,1}\big(\overline{\Omega}\times \overline{\Omega}\big), 
\end{equation}
where $C^{0,1}\big(\overline{\Omega}\times \overline{\Omega}\big)$ denotes the set of Lipschitz continuous functions on $\overline{\Omega}\times \overline{\Omega}$.

Under the previous assumption \eqref{Regularity}, for $u^0\in L^1(\Omega)$, we know from~\cite{Coville2013a} that there exists a unique non-negative classical solution in $C([0,+\infty), L^1(\Omega)) \cap C^1((0,+\infty), C^{2,\alpha} (\Omega))$, which we denote $t \longmapsto S_t u^0$.

It will also be convenient to introduce the space $Z := \{u \in C^{2,\alpha}(\Omega), \; u\geq 0\}$.
Finally, we define the non-negative function $H(w):= w-1-\ln(w)$ for $w>0$, and for $u$ in $Z$ 
\begin{equation}
V(u):= \int_\Omega \left(u(x)-1-\ln(u(x)\right) \, dx,
\end{equation}
the last integral possibly being equal to $+\infty$.

Our result is then the following: 
\smallskip
\begin{theorem}
Assume \eqref{Positive}, \eqref{Regularity}, \eqref{SumToOne}. Then for any initial datum $u^0$ in $L^1(\Omega)$, $u^0\geq 0$, $u^0 \neq 0$, the solution to \eqref{Model} satisfies 
\[u(t,\cdot) \longrightarrow 1\]
uniformly in $\Omega$.
\end{theorem}

\begin{proof}
\textit{First step: computation of the Lyapunov functional.}

First, let us remark that by the parabolic strong maximum principle, $u(t,x)>0$ for all $t>0$, $x\in \Omega$. Now, let us check that this holds true also for $x\in \partial \Omega$, from which we will infer that $V(u(t))$ is well defined for all $t>0$. By the parabolic strong maximum principle at the boundary, we have the following alternative for $x \in \partial \Omega$: either $u(t,x) > 0$ or $u(t,x) = 0$ and $\dfrac{\partial u}{\partial n} (t,x) < 0$. Only $u(t,x) > 0$ can hold due to the Neumann boundary conditions. 

We now consider $g(t):= V(u(t))$ for $t>0$, where $\{u(t)\}_{t\geq0}$ is the trajectory given rise to by $u_0$. Let us prove that this is a Lyapunov functional, by computing for $t>0$
\begin{align*}
g'(t)& = \int_\Omega \dfrac{\partial u}{\partial t}(t) \left(1- \frac{1}{u(t)}\right) \\
			 & = \int_\Omega \Delta u (t)  \left(1- \frac{1}{u(t)}\right) - \mu \int_{\Omega}  (1- K[u(t)]) \left(1-u(t))\right) \\ 
			 & = - \int_\Omega \frac{\left| \nabla{\left(u(t,x)\right)}\right|^2}{u^2(t,x)}  \, dx
 -\mu \int_{\Omega^2} K(x,y) \left(1-u(t,x))\right) \left(1-u(t,y)\right) \, dx \, dy.
\end{align*}
after integration by part for the first term. For the second one, we used $1-K[u] = K[1-u]$, owing to \eqref{SumToOne}.

Thanks to \eqref{Positive}, this yields $g'(t) \leq 0$ \textit{i.e.}, that $g$ is non-increasing over the real line. 
Since $g\geq 0$, we infer the convergence of $g(t)$ as $t$ tends to $+\infty$, and we denote $l$ its limit. \par

\textit{Second step: compactness of trajectories.}

Since $C^{2,\alpha}(\Omega)$ is compactly embedded into $C\big(\overline{\Omega}\big)$, the trajectory $\{S_t u^0\}_{t\geq0}$ is relatively compact in $C(\overline{\Omega})$, meaning that one can find $\bar u \geq 0$ in $C(\overline{\Omega})$ and a sequence $(t_k)$ tending to $+\infty$ in $k$, such that $u(t_k)$ converges to $\bar u$ as $k$ goes to $+\infty$, in $C\big(\overline{\Omega}\big)$. Note that the limit cannot be identically $0$ since otherwise $g(t)$ would go to $+\infty$, in contradiction with its convergence to $l$. 

Our aim is to prove that $\bar u = 1$, which will mean that the whole trajectory converges to $\bar u$, hence the expected result.

\textit{Third step: identifying the limit.}

Let us now consider the trajectory starting from the initial datum $\bar u$, namely $\{S_t\bar u\}_{t\geq0}$, which we also denote by $\{\tilde{u}(t)\}_{t\geq0}$. Because $\bar u \geq 0$, $\bar u \neq 0$, we again have $\tilde{u}(t,x)>0$ for all $t>0$, $x\in \overline{\Omega}$. Let us prove that $V$ is constant along the trajectory $\{S_t\bar u\}_{t\geq0}$ for $t>0$.  \par 
For this, we write $V(\tilde{u}(t)) = V(S_t \bar u) = V\left(S_t \lim_{k \rightarrow +\infty} S_{t_k} u^0\right) = V\left( \lim_{k \rightarrow +\infty} S_{t+t_k} u^0\right)$. It is also easy to see that for any $u$ in $C\big(\overline{\Omega}\big)$ which is furthermore positive on $\overline{\Omega}$, $V$ (seen as acting on $C\big(\overline{\Omega}\big)$) is continuous at $u$, and this implies $V(\tilde{u}(t)) = \lim_{k \rightarrow +\infty} V\big(S_{t+t_k} u^0\big) = l$.
As claimed the function $t\longmapsto V(S_t \bar u)$ is constant (equal to $l$) for $t>0$. 

Hence its derivative must be zero for $t>0$: from the computations made in the first step, it must hold that both $\int_\Omega \frac{\left| \nabla{\left(\tilde{u}(t)\right)}\right|^2}{\tilde{u}^2(t)}$ 
and $\int_{\Omega^2} K(x,y) \left(\tilde{u}(t,x) - 1)\right) \left(\tilde{u}(t,y) - 1\right) \, dx \, dy$ vanish identically for $t>0$. Let us now fix $t>0$, and from the first term, we know that $\tilde{u}(t)$ is a constant. From the second term and owing to $K>0$, this constant must be equal to $1$. By continuity of the trajectory, this also holds true at $t=0$, \textit{i.e.}, $\bar u = 1$, which ends the proof. 

\end{proof}

{
\bibliography{BibliographyNoteKPP}
\bibliographystyle{acm}}

\end{document}